\documentclass[11pt,reqno]{amsart}
\usepackage[T1]{fontenc}
\usepackage{lmodern}
\usepackage[a4paper,margin=28mm]{geometry}
\usepackage{amsmath,amssymb,amsthm,mathtools}
\usepackage{mathrsfs}
\usepackage[expansion=false]{microtype}
\usepackage{enumitem}
\usepackage{cite}
\usepackage{needspace}
\usepackage[hidelinks]{hyperref}
\hypersetup{pdftitle={Nonfinitely based intervals of power semiring varieties},
	pdfauthor={Xiaolei Shao and Zidong Gao}
 pdfsubject={Power semirings, variety intervals, and the finite basis problem}}

\newtheorem{theorem}{Theorem}[section]
\newtheorem{lemma}[theorem]{Lemma}
\newtheorem{corollary}[theorem]{Corollary}
\newtheorem{proposition}[theorem]{Proposition}
\theoremstyle{definition}
\newtheorem{definition}[theorem]{Definition}

\numberwithin{equation}{section}

\newcommand{\Pow}{\mathcal{P}}
\newcommand{\Full}[1]{\mathcal{P}(#1)^0}
\newcommand{\Var}{\mathsf{V}}
\newcommand{\Id}{\operatorname{Id}}
\newcommand{\Hom}{\operatorname{Hom}}
\newcommand{\occ}{\operatorname{occ}}
\newcommand{\expon}{\operatorname{exp}}
\newcommand{\Z}{\mathbb{Z}}
\newcommand{\F}{\mathbb{F}}
\newcommand{\E}{\mathbb{E}}
\newcommand{\Prb}{\mathbb{P}}
\newcommand{\bu}{\mathbf{u}}
\newcommand{\bv}{\mathbf{v}}
\newcommand{\bw}{\mathbf{w}}
\newcommand{\bq}{\mathbf{q}}
\newcommand{\bp}{\mathbf{p}}
\newcommand{\bt}{\mathbf{t}}
\newcommand{\calI}{\mathcal{I}}
\newcommand{\calC}{\mathcal{C}}
\newcommand{\calE}{\mathcal{E}}
\newcommand{\calG}{\mathcal{G}}
\newcommand{\calU}{\mathcal{U}}
\newcommand{\calV}{\mathcal{V}}
\newcommand{\calW}{\mathcal{W}}
\renewcommand{\leq}{\leqslant}
\renewcommand{\geq}{\geqslant}

\title[Intervals of power semiring varieties]{Nonfinitely based intervals of power semiring varieties}
\author{Xiaolei Shao}
\author{Zidong Gao}
\date{}

\begin{document}
\begin{abstract}
We study intervals in the lattice of additively idempotent semiring varieties associated with power semirings of finite groups. We prove that every variety lying above the nonempty power semiring of a finite group of order at least three and below a variety generated by finitely many full power semirings of finite groups is nonfinitely based. In fact, none of these varieties admits an identity basis with a fixed finite bound on the number of variables. In particular, adjoining the empty set gives an interval consisting entirely of nonfinitely based varieties. We also construct equational upper bounds that are closed under zero adjunction and yield further intervals with the same property. The proof combines finite commutative quotient semirings, a cardinality estimate for kernel blockers over finite modules, and a uniform estimate for fibres of ordered products in finite groups. As a consequence, the full power semiring of a finite group is finitely based precisely when the group has at most two elements.
\end{abstract}
\maketitle

\section{Introduction}
An \emph{additively idempotent semiring}, or \emph{ai-semiring}, is an algebra $(S,+,\cdot)$ whose additive reduct is a commutative idempotent semigroup, whose multiplicative reduct is a semigroup, and in which multiplication distributes over addition on both sides. For an ai-semiring $S$, let $\Var(S)$ denote the variety it generates. A variety is \emph{finitely based} if its identities admit a finite basis, and is \emph{nonfinitely based} otherwise. An algebra is finitely based if the variety it generates is finitely based.

For a semigroup $S$, let $\Pow(S)$ denote the set of all nonempty subsets of $S$, with operations
\[
 A+B=A\cup B,\qquad AB=\{ab:a\in A,\ b\in B\}.
\]
This ai-semiring is called the \emph{power semiring} of $S$. If the empty set is included, the resulting semiring is denoted by $\Full{S}$. More generally, for an ai-semiring $S$, the semiring $S^0=S\cup\{0\}$ is obtained by adjoining a new element satisfying
\[
 a+0=0+a=a,\qquad a0=0a=0\qquad(a\in S^0).
\]
Thus the notation $\Full{S}$ agrees with zero adjunction to $\Pow(S)$. The signature consists of the two binary operations $+$ and $\cdot$; expansions by constants are considered in Appendix~\ref{app:constants}.

The finite basis problem for power semirings was raised by Dolinka~\cite{d}. Gusev and Volkov~\cite{gv} obtained nonfinite-basis results for power semirings of finite nonabelian solvable groups, and Dolinka, Gusev and Volkov~\cite{dgv} extended this line of investigation to inverse semigroups. These results are closely related to the study of identities of ai-semirings and their multiplicative reducts in~\cite{jrz}.

Gao, Ren, Shao and Yue~\cite{grsy} proved that $\Pow(G)$ is nonfinitely based for a finite group $G$ if and only if $|G|\geq 3$. Their result is part of an interval theorem: every variety in
\[
 [\Var(S_7),\Var(\Pow(G))]
\]
is nonfinitely based when $|G|\geq 3$~\cite[Corollary~3.5]{grsy}. Here $S_7$ denotes the three-element ai-semiring used in that paper. The effect of zero adjunction on identities and on intervals of varieties is studied in~\cite{wrz,grz}.

In this paper we obtain intervals whose upper endpoints are generated by \emph{full} power semirings and whose lower endpoints are generated by \emph{nonempty} power semirings. For varieties $\calU\subseteq\calW$, write
\[
 [\calU,\calW]=\{\calV:\calU\subseteq\calV\subseteq\calW\}.
\]
A variety has \emph{infinite axiomatic rank} if, for every positive integer $k$, its identities involving at most $k$ variables do not form a basis for all its identities. In particular, such a variety is nonfinitely based.

\begin{theorem}\label{thm:interval}
Let $G_1,\ldots,G_t$ be finite groups, where $t\geq 1$, and put
\[
 \calW=\Var\bigl(\Full{G_1},\ldots,\Full{G_t}\bigr).
\]
Let $H$ be a finite group with $|H|\geq 3$. Then every ai-semiring variety $\calV$ satisfying
\begin{equation}\label{eq:maininterval}
 \Var(\Pow(H))\subseteq\calV\subseteq\calW
\end{equation}
has infinite axiomatic rank and is nonfinitely based.
\end{theorem}

The two endpoints in~\eqref{eq:maininterval} play different roles: the lower endpoint supplies finite quotient models of identities, while the upper endpoint supplies identities separating these quotients. This yields the following zero-adjunction interval.

\begin{corollary}\label{cor:zerointerval}
Let $G$ be a finite group with $|G|\geq 3$. Then every variety in
\[
 [\Var(\Pow(G)),\Var(\Full{G})]
\]
has infinite axiomatic rank and is nonfinitely based.
\end{corollary}

\begin{corollary}\label{thm:classification}
Let $G$ be a finite group. Then $\Full{G}$ is finitely based if and only if $|G|\leq 2$.
\end{corollary}

We also obtain intervals of the form $[\calU,\calU^0]$, where $\calU^0$ is the variety generated by the zero extensions of members of $\calU$. For each prescribed finite family of power semirings, we further construct an equationally defined, zero-closed upper endpoint containing the given generators. Every variety between the lower power-semiring variety and this endpoint has infinite axiomatic rank. These results are stated in Section~\ref{sec:extensions}.

The occurrence construction for the cyclic group of order three in~\cite{note} provides the starting point of the proof. We use occurrences of the lower word together with auxiliary points for variables appearing only in the upper term. This gives quotient models for identities of \emph{nonempty} power semirings. A strict cardinality estimate for difference submodules extends the construction to arbitrary finite abelian groups of order at least three. Ordered products and exponent lifting then yield separating identities for arbitrary finite groups.

Section~\ref{sec:prelim} collects the notation and local criteria used in the proof. Sections~\ref{sec:modules}--\ref{sec:sparse} establish the module, quotient, and counting arguments. The main theorem is proved in Section~\ref{sec:separation}. Further intervals and examples are given in Sections~\ref{sec:extensions} and~\ref{sec:examples}.

\section{Preliminaries}\label{sec:prelim}
Let $X$ be a countably infinite set of variables, and let $X^+$ be the free semigroup over $X$. Every ai-semiring term can be written as a finite nonempty sum of words from $X^+$. We use bold lowercase letters $\bu,\bv,\bw,\ldots$ for terms, and ordinary lowercase letters $x,y,z,\ldots$ for variables. The \emph{content} $c(\bu)$ is the set of variables occurring in $\bu$. For a word $\bq$, the number of occurrences of $x$ in $\bq$ is denoted by $\occ(x,\bq)$.

The additive order is defined by $a\leq b$ if $a+b=b$. Following~\cite{grsy}, we write
\[
 \bu\preceq\bv
\]
for the identity $\bv\approx\bv+\bu$, and call this an \emph{ai-semiring inequality}. For an ai-semiring $S$, let $\Id(S)$ denote its identities, and let $\Id_k(S)$ denote those involving only $x_1,\ldots,x_k$. The same notation is used for varieties. All equivalences of identities in this paper are taken relative to the ai-semiring axioms.

\begin{lemma}\label{lem:reduction}
Every finite set $\Sigma$ of ai-semiring identities is equivalent to a finite set $\calE$ of inequalities
\[
 \bq\preceq\bu,
\]
where $\bq$ is a word and $\bu$ is a finite nonempty sum of words.
\end{lemma}
\begin{proof}
For an identity $\bu_1+\cdots+\bu_a\approx\bv_1+\cdots+\bv_b$ with word summands, take the inequalities
\[
 \bu_i\preceq\bv_1+\cdots+\bv_b\quad(1\leq i\leq a),
 \qquad
 \bv_j\preceq\bu_1+\cdots+\bu_a\quad(1\leq j\leq b).
\]
The two families express the two directions of the original equality in the additive order. Taking their union over $\Sigma$ proves the assertion.
\end{proof}

The next observation records the role of variables under zero adjunction; it is the word-inequality form of~\cite[Proposition~1.5]{wrz}. If $\bu=\bu_1+\cdots+\bu_s$ has word summands and $\bq\in X^+$, let $D_{\bq}(\bu)$ be the sum of those $\bu_i$ for which $c(\bu_i)\subseteq c(\bq)$. When no summand has this property, write $D_{\bq}(\bu)=\varnothing$.

\begin{lemma}[Zero adjunction]\label{lem:zeroidentity}
Let $S$ be an ai-semiring, let $\bq$ be a word, and let $\bu$ be a term. Then
\[
 S^0\models\bq\preceq\bu
 \quad\Longleftrightarrow\quad
 D_{\bq}(\bu)\ne\varnothing
 \ \text{ and }\ S\models\bq\preceq D_{\bq}(\bu).
\]
\end{lemma}
\begin{proof}
Assign the new zero to all variables outside $c(\bq)$ and assign elements of $S$ to those in $c(\bq)$. The surviving upper summands are exactly those in $D_{\bq}(\bu)$. The value of $\bq$ belongs to $S$, so at least one upper summand survives, proving necessity.

Conversely, consider an assignment into $S^0$. If a variable of $\bq$ is assigned zero, the lower word evaluates to zero and the inequality holds. Otherwise all variables of $D_{\bq}(\bu)$ receive elements of $S$, and the assumed inequality gives
\[
 \bq\leq D_{\bq}(\bu)\leq\bu
\]
under this assignment.
\end{proof}

\begin{lemma}\label{lem:finitevariable}
Let $A$ be a finite ai-semiring and let $k\geq 1$. There is a finite subset $\Delta_k\subseteq\Id_k(A)$ from which every member of $\Id_k(A)$ follows.
\end{lemma}
\begin{proof}
Let $F_k$ be the finite set of $k$-ary term functions of $A$. For each $f\in F_k$, choose a representative term $\bt_f$ over $x_1,\ldots,x_k$. Let $p_i$ be the $i$th projection. Include in $\Delta_k$ the identities
\[
 x_i\approx\bt_{p_i},\qquad
 \bt_f+\bt_g\approx\bt_{f+g},\qquad
 \bt_f\bt_g\approx\bt_{fg}
 \quad(1\leq i\leq k,\ f,g\in F_k),
\]
where operations on term functions are pointwise. These identities hold in $A$. Induction on term formation shows that each term is equivalent, modulo $\Delta_k$, to the chosen representative of its term function. Consequently, every identity in $\Id_k(A)$ follows from $\Delta_k$.
\end{proof}

\begin{lemma}\label{lem:local}
Let $\calU\subseteq\calW$ be ai-semiring varieties. Suppose that, for every $k\geq 1$, there is an ai-semiring $Q_k$ such that
\begin{enumerate}[label=\textup{(\roman*)},leftmargin=2.4em]
\item every subsemiring of $Q_k$ generated by at most $k$ elements belongs to $\calU$;
\item $Q_k\notin\calW$.
\end{enumerate}
Then every variety in $[\calU,\calW]$ has infinite axiomatic rank.
\end{lemma}
\begin{proof}
Let $\calV\in[\calU,\calW]$ and let $k\geq 1$. Any assignment of at most $k$ variables into $Q_k$ has its image in a subsemiring belonging to $\calU$, and hence to $\calV$. Thus $Q_k$ satisfies every identity of $\calV$ involving at most $k$ variables. Since $Q_k\notin\calW$, it does not belong to $\calV$. Those identities therefore do not form a basis for $\calV$.
\end{proof}

\section{Kernel blockers over finite modules}\label{sec:modules}
Fix a finite abelian group $H$ with $|H|\geq 3$, write its operation additively, and put
\[
 n=\expon(H),\qquad R=\Z/n\Z.
\]
Thus $n\geq 2$, and every abelian group of exponent dividing $n$ is naturally an $R$-module. In this section all modules are finite.

\begin{lemma}\label{lem:characters}
Let $M$ be a finite $R$-module.
\begin{enumerate}[label=\textup{(\roman*)},leftmargin=2.4em]
\item Every homomorphism from a submodule of $M$ to $R$ extends to a homomorphism $M\to R$.
\item Homomorphisms $M\to R$ separate the elements of $M$, and $|\Hom_R(M,R)|=|M|$.
\item Every submodule of $M$ isomorphic to $R^s$ is a direct summand of $M$.
\end{enumerate}
\end{lemma}
\begin{proof}
For (i), suppose that $\chi:N\to R$ is defined on a submodule $N$, and choose $x\in M\setminus N$. Let $\ell$ be the order of $x$ and $d$ the order of $x+N$ in $M/N$. Then $d\mid\ell\mid n$ and $dx\in N$. The element $a=\chi(dx)$ is annihilated by $\ell/d$, so in $\Z/n\Z$ it is a multiple of $nd/\ell$ and, consequently, belongs to $dR$. Choose $b\in R$ with $db=a$. The rule
\[
 \chi'(y+kx)=\chi(y)+kb\qquad(y\in N,\ k\in\Z)
\]
is well-defined and extends $\chi$ to $N+\langle x\rangle$. Iteration proves (i).

If $0\ne x\in M$ has order $d$, the homomorphism $\langle x\rangle\to R$ sending $x$ to $n/d$ is nonzero at $x$; apply (i). For the counting assertion, decompose $M$ into cyclic groups of orders dividing $n$ and use $|\Hom(C_d,C_n)|=d$.

Finally, if $F\leq M$ has basis $v_1,\ldots,v_s$ over $R$, extend its coordinate maps $F\to R$ to maps $\chi_i:M\to R$ using (i). The map
\[
 M\longrightarrow F,\qquad x\longmapsto\sum_{i=1}^{s}\chi_i(x)v_i
\]
is a retraction, proving (iii).
\end{proof}

We shall use the following property of nonzero elements of $H$.

\begin{lemma}\label{lem:nonzero}
Suppose $\lambda_0,\ldots,\lambda_s\in R$ satisfy
\[
 \sum_{i=0}^{s}\lambda_i=1
\]
and the coefficient vector is not a standard basis vector. Then there exist $y_0,\ldots,y_s\in H\setminus\{0\}$ such that
\begin{equation}\label{eq:nonzero_relation}
 \sum_{i=0}^{s}\lambda_i y_i=0.
\end{equation}
\end{lemma}
\begin{proof}
A nonunit $a\in R$ annihilates some nonzero element of $H$: indeed, $H$ contains a cyclic subgroup of order $n$, and multiplication by $a$ on that subgroup has a nontrivial kernel. Also, if $a\ne 0$ in $R$, then $ah\ne 0$ for some $h\in H$, by the definition of the exponent.

For every integer $k\geq 2$, one can choose $k$ nonzero elements of $H$ with sum zero. For $k=2$ use $h,-h$. For $k=3$ choose nonzero $a,b$ with $b\ne-a$, and use $a,b,-a-b$; such a choice exists because $|H|\geq 3$. The general case follows by writing $k$ as a sum of twos and threes.

If none of the $\lambda_i$ is a unit, choose each $y_i\ne 0$ in the kernel of multiplication by $\lambda_i$. If at least two coefficients are units, choose nonzero values with sum zero for the products $\lambda_i y_i$ at the unit positions, and choose kernel elements at every other position.

It remains to consider exactly one unit coefficient, say $\lambda_j$. Some other coefficient $\lambda_k$ is nonzero; otherwise the normalization would make the coefficient vector a standard basis vector. Choose $y_k$ with $\lambda_k y_k\ne 0$, set $y_j=-\lambda_j^{-1}\lambda_k y_k$, and choose nonzero kernel elements at the remaining positions. These choices satisfy \eqref{eq:nonzero_relation}.
\end{proof}

\begin{definition}
A nonempty subset $C$ of an $R$-module $M$ is an \emph{$H$-kernel blocker} if
\[
 C\cap\ker f\ne\varnothing\qquad\text{for every }f\in\Hom_R(M,H).
\]
For a nonempty finite set $C\subseteq M$, define its \emph{difference submodule} by
\[
 L(C)=\langle c-c':c,c'\in C\rangle_R.
\]
For any $c_0\in C$, its affine hull is the coset $c_0+L(C)$.
\end{definition}

\begin{proposition}[Strict affine cardinality defect]\label{prop:defect}
If $C\subseteq M\setminus\{0\}$ is an $H$-kernel blocker, then
\begin{equation}\label{eq:defect}
 |L(C)|<n^{|C|-1}.
\end{equation}
\end{proposition}
\begin{proof}
Write $C=\{c_0,\ldots,c_s\}$. The module $L=L(C)$ is generated by the $s$ differences $v_i=c_i-c_0$, so $|L|\leq n^s$. Suppose equality holds. Then the natural surjection $R^s\to L$ is an isomorphism, and the $v_i$ form an $R$-basis of $L$.

First suppose $c_0\notin L$. By Lemma~\ref{lem:characters}, there is a homomorphism $M/L\to R$ that is nonzero on $c_0+L$. Embed $R$ into $H$ as a cyclic subgroup of order $n$. The resulting map $M\to H$ has the same nonzero value on every $c_i$, contradicting the blocker property.

Suppose instead that $c_0\in L$, and write $c_0=\sum_{i=1}^{s}a_i v_i$. Put
\[
 \lambda_0=1+\sum_{i=1}^{s}a_i,\qquad \lambda_i=-a_i\quad(1\leq i\leq s).
\]
Then
\begin{equation}\label{eq:affinerel}
 \sum_{i=0}^{s}\lambda_i=1,\qquad \sum_{i=0}^{s}\lambda_i c_i=0.
\end{equation}
The vector $(\lambda_0,\ldots,\lambda_s)$ cannot be a standard basis vector, since that would force one of the $c_i$ to be zero. Lemma~\ref{lem:nonzero} supplies nonzero $y_i\in H$ satisfying the corresponding relation.

Define $f_0:L\to H$ on its basis by $f_0(v_i)=y_i-y_0$. Equation~\eqref{eq:affinerel} for the $y_i$ gives $f_0(c_0)=y_0$, and hence $f_0(c_i)=y_i$ for every $i$. By Lemma~\ref{lem:characters}(iii), $L$ is a direct summand of $M$, so $f_0$ extends to a homomorphism $M\to H$. This map is nonzero on all of $C$, again a contradiction. Therefore equality is impossible, proving \eqref{eq:defect}.
\end{proof}

When $n=p$ is prime, \eqref{eq:defect} says that $C$ is affinely dependent over $\F_p$. For composite $n$, the cardinality estimate gives the corresponding bound for counting copies of $C$ in powers of $R$.

\section{Occurrence configurations and quotient semirings}\label{sec:quotients}
Fix a finite abelian group $H$ with $|H|\geq 3$, and put $n=\expon(H)$ and $R=\Z/n\Z$. Let $\calE$ be a finite family of inequalities $\bq\preceq\bu$ valid in $\Pow(H)$, where $\bq$ is a word. For an abelian group $M$, multiplication in $\Pow(M)$ is Minkowski addition, denoted by $+_M$ when it is necessary to distinguish it from union.

For $\varepsilon=(\bq\preceq\bu)\in\calE$, write
\[
 c(\bq)\cup c(\bu)=\{x_1,\ldots,x_\ell\},\qquad
 t_i=\occ(x_i,\bq),\qquad s_i=\max\{1,t_i\}.
\]
Let $W_\varepsilon$ be the free $R$-module with basis
\[
 e_{i,j}\qquad(1\leq i\leq\ell,\ 1\leq j\leq s_i),
\]
and set
\begin{equation}\label{eq:occurrence}
 Y_i=\{e_{i,1},\ldots,e_{i,s_i}\},\qquad
 \alpha_\varepsilon=\sum_{i=1}^{\ell}\sum_{j=1}^{t_i}e_{i,j}.
\end{equation}
Thus a variable absent from $\bq$ has one auxiliary basis vector in $Y_i$ and contributes nothing to $\alpha_\varepsilon$. Define the \emph{occurrence blocker}
\begin{equation}\label{eq:blocker}
 B_\varepsilon=\alpha_\varepsilon-\bu(Y_1,\ldots,Y_\ell)
 \subseteq W_\varepsilon.
\end{equation}
Every $Y_i$ is nonempty, so the upper evaluation and $B_\varepsilon$ are nonempty.

\begin{lemma}\label{lem:occurrence}
The set $B_\varepsilon$ is an $H$-kernel blocker in $W_\varepsilon$. If $0\in B_\varepsilon$, then $\bq\preceq\bu$ holds in $\Pow(M)$ for every $R$-module $M$.
\end{lemma}
\begin{proof}
Let $f:W_\varepsilon\to H$ be an $R$-homomorphism. Since the inequality holds in $\Pow(H)$ and the sets $f(Y_i)$ are nonempty,
\[
 f(\alpha_\varepsilon)
 \in\bq(f(Y_1),\ldots,f(Y_\ell))
 \subseteq\bu(f(Y_1),\ldots,f(Y_\ell))
 =f\bigl(\bu(Y_1,\ldots,Y_\ell)\bigr).
\]
Hence $f$ vanishes on a member of $B_\varepsilon$.

Suppose that $0\in B_\varepsilon$. Let $X_1,\ldots,X_\ell$ be nonempty subsets of an $R$-module $M$, and choose $p\in\bq(X_1,\ldots,X_\ell)$. Express $p$ by choosing an element of $X_i$ at each occurrence of $x_i$ in $\bq$, and send $e_{i,j}$ to the corresponding element. If $t_i=0$, send the auxiliary vector $e_{i,1}$ to any element of $X_i$. The resulting homomorphism $\psi:W_\varepsilon\to M$ satisfies
\begin{equation}\label{eq:occurrencemap}
 \psi(\alpha_\varepsilon)=p,\qquad \psi(Y_i)\subseteq X_i.
\end{equation}
Since $\alpha_\varepsilon\in\bu(Y_1,\ldots,Y_\ell)$, it follows that
\[
 p\in\psi\bigl(\bu(Y_1,\ldots,Y_\ell)\bigr)
 =\bu(\psi(Y_1),\ldots,\psi(Y_\ell))
 \subseteq\bu(X_1,\ldots,X_\ell).
\]
This proves the inequality in $\Pow(M)$.
\end{proof}

Retain the blockers $B_\varepsilon$ with $0\notin B_\varepsilon$. For a finite $R$-module $V$, a \emph{forbidden configuration} is a set of the form
\begin{equation}\label{eq:forbidden}
 a+\phi(B_\varepsilon),\qquad
 a\in V,\quad\phi\in\Hom_R(W_\varepsilon,V),\quad
 0\notin\phi(B_\varepsilon).
\end{equation}
Let $\calI_{\calE}(V)$ be the family of all nonempty subsets of $V$ that contain such a configuration. Define an equivalence relation $\theta_{\calE}$ on $\Pow(V)$ by
\[
 X\mathrel{\theta_{\calE}}Y
 \quad\Longleftrightarrow\quad
 X=Y\ \text{ or }\ X,Y\in\calI_{\calE}(V).
\]

\begin{lemma}\label{lem:congruence}
The relation $\theta_{\calE}$ is a semiring congruence. The quotient
\[
 Q_{\calE}(V)=\Pow(V)/\theta_{\calE}
\]
is a finite commutative ai-semiring.
\end{lemma}
\begin{proof}
The family $\calI_{\calE}(V)$ is upward closed and invariant under translations. If $X\in\calI_{\calE}(V)$ and $Z\in\Pow(V)$, then $X\cup Z$ belongs to the family. Choosing $z\in Z$, we also have $X+z\subseteq X+_V Z$, so $X+_V Z$ belongs to the family. These facts prove compatibility with addition and multiplication. The quotient is finite and commutative because $V$ is finite and abelian.
\end{proof}

\begin{lemma}\label{lem:fragmentmodel}
For every finite $R$-module $V$, the semiring $Q_{\calE}(V)$ satisfies every inequality in $\calE$.
\end{lemma}
\begin{proof}
Fix $\varepsilon=(\bq\preceq\bu)\in\calE$ and nonempty representatives $X_1,\ldots,X_\ell$ for an assignment into the quotient. If $\bq(X)\subseteq\bu(X)$, the inequality holds under this assignment.

Otherwise choose $p\in\bq(X)\setminus\bu(X)$ and construct $\psi$ as in~\eqref{eq:occurrencemap}, including the choices for auxiliary vectors. Then
\[
 \psi(\alpha_\varepsilon)=p,\qquad
 \psi\bigl(\bu(Y)\bigr)\subseteq\bu(X).
\]
Lemma~\ref{lem:occurrence} implies $0\notin B_\varepsilon$. Moreover, $0\notin\psi(B_\varepsilon)$: an equality $\psi(\alpha_\varepsilon-b)=0$ with $b\in\bu(Y)$ would imply $p\in\bu(X)$. Thus
\begin{equation}\label{eq:forcing}
 p+(-\psi)(B_\varepsilon)
 =p-\psi(B_\varepsilon)
 =\psi\bigl(\bu(Y)\bigr)\subseteq\bu(X)
\end{equation}
is a forbidden configuration. Both $\bu(X)$ and $\bu(X)\cup\bq(X)$ belong to $\calI_{\calE}(V)$ and therefore have the same quotient class. This proves the inequality.
\end{proof}

Two nonempty finite sets have the same \emph{affine type} if an isomorphism between their difference submodules, followed by a translation, maps one set onto the other. More explicitly, a copy of $C$ in a module $V$ is a set $a+\iota(C-c_0)$, where $c_0\in C$, $a\in V$, and $\iota:L(C)\to V$ is an injective homomorphism.

\begin{lemma}\label{lem:types}
For the fixed family $\calE$, only finitely many affine types of sets $\phi(B_\varepsilon)$ with $0\notin\phi(B_\varepsilon)$ occur as the finite target module varies. Every representative $C$ satisfies
\[
 |L(C)|<n^{|C|-1}.
\]
\end{lemma}
\begin{proof}
For a fixed $W_\varepsilon$, two maps with the same kernel induce an isomorphism between their images carrying one image of $B_\varepsilon$ onto the other. There are finitely many submodules of $W_\varepsilon$ and finitely many retained blockers, which proves finiteness of the types.

The set $\phi(B_\varepsilon)$ is an $H$-kernel blocker in $\phi(W_\varepsilon)$: compose any homomorphism from that image module to $H$ with $\phi$ and apply Lemma~\ref{lem:occurrence}. Proposition~\ref{prop:defect} gives the inequality. Both cardinalities are invariant under affine isomorphism.
\end{proof}

\section{Sparse avoidance and ordered product fibres}\label{sec:sparse}
We now choose sparse subsets that avoid the configurations of Section~\ref{sec:quotients} and meet a prescribed family of dense sets.

\begin{lemma}[Sparse hitting and avoidance]\label{lem:sparse}
Let $n\geq 2$, let $R=\Z/n\Z$, and let $\calC$ be a finite family of nonempty configurations in finite $R$-modules, each satisfying
\[
 |L(C)|<n^{|C|-1}.
\]
Fix constants $\eta>0$, $b\geq 1$, and $a\geq 1$. For each positive integer $r$, let $\mathcal{T}_r$ be a nonempty family of at most $ab^r$ subsets of $R^r$, each of cardinality at least $\eta n^r$. Then, for all sufficiently large $r$, there exists
\[
 D\subseteq R^r\setminus\{0\},\qquad |D|\leq 2r^2,
\]
which meets every member of $\mathcal{T}_r$ and contains no affine copy of any member of $\calC$.
\end{lemma}
\begin{proof}
Write $N=n^r$ and select each nonzero point of $R^r$ independently with probability $p=r^2/N$, obtaining a random set $D_r$. For sufficiently large $r$, $p<1$. The probability of missing at least one member of $\mathcal{T}_r$ is at most
\begin{equation}\label{eq:hittingprob}
 ab^r(1-p)^{\eta N-1}
 \leq ab^r\exp\bigl(-p(\eta N-1)\bigr)=o(1).
\end{equation}

Fix $C\in\calC$, choose $c_0\in C$, and put $h=|C|$ and $L=L(C)$. An affine copy of $C$ in $R^r$ has the form
\[
 z+\iota(C-c_0),
\]
where $z\in R^r$ and $\iota:L\to R^r$ is an injective homomorphism. By Lemma~\ref{lem:characters}(ii), the number of such copies is at most
\[
 N|\Hom_R(L,R^r)|=N|L|^r.
\]
A copy containing zero is never contained in $D_r$; every other copy is contained in $D_r$ with probability $p^h$. Hence the expected number of copies is at most
\begin{equation}\label{eq:copyprob}
 N|L|^r p^h
 =r^{2h}\left(\frac{|L|}{n^{h-1}}\right)^r=o(1).
\end{equation}
Summing over the fixed finite family and applying Markov's inequality shows that the probability of containing a forbidden copy is $o(1)$.

Finally, $\E|D_r|<r^2$ and $\operatorname{Var}(|D_r|)\leq r^2$. Chebyshev's inequality therefore gives
\[
 \Prb(|D_r|>2r^2)\leq r^{-2}.
\]
Together with \eqref{eq:hittingprob} and \eqref{eq:copyprob}, this shows that the probability of any failure tends to zero. The required set exists for every sufficiently large $r$.
\end{proof}

We next produce the dense sets to which the hitting condition will be applied. Products are taken in the displayed order.

\begin{lemma}[Uniform product-fibre bound]\label{lem:fibres}
Let $G$ be a finite group and let $m\geq 2$ be divisible by $\expon(G)$. For $g_1,\ldots,g_r\in G$ and $g\in\langle g_1\rangle\cdots\langle g_r\rangle$, the number of tuples $(k_1,\ldots,k_r)\in\{1,\ldots,m\}^r$ satisfying
\[
 g_1^{k_1}\cdots g_r^{k_r}=g
\]
is at least
\begin{equation}\label{eq:fibredensity}
 m^{-(|G|-1)}m^r.
\end{equation}
\end{lemma}
\begin{proof}
Choose each $k_i$ independently and uniformly from $\{1,\ldots,m\}$. Then $g_i^{k_i}$ is uniform on $K_i=\langle g_i\rangle$. Let $\mu_i$ be the distribution of the product of the first $i$ factors, and put
\[
 S_0=\{1_G\},\qquad S_i=K_1\cdots K_i.
\]
Each $S_i$ is the support of $\mu_i$, and $S_{i-1}\subseteq S_i$ because $1_G\in K_i$. Let $\delta_i$ be the smallest positive value of $\mu_i$, with $\delta_0=1$. Convolution gives
\begin{equation}\label{eq:convolution}
 \mu_i(x)=\frac{1}{|K_i|}\sum_{h\in K_i}\mu_{i-1}(xh^{-1}).
\end{equation}
If $S_i=S_{i-1}$, then $S_{i-1}K_i=S_{i-1}$; for each $x\in S_i$, every summand in \eqref{eq:convolution} is at least $\delta_{i-1}$. Thus $\delta_i\geq\delta_{i-1}$. If the support grows strictly, at least one summand is positive, so
\[
 \delta_i\geq\delta_{i-1}/|K_i|\geq\delta_{i-1}/m.
\]
There are at most $|G|-1$ strict support increases. Consequently $\delta_r\geq m^{-(|G|-1)}$. Multiplying this probability bound by $m^r$ proves \eqref{eq:fibredensity}.
\end{proof}

Fix a finite family $\calG=\{G_1,\ldots,G_t\}$ and an integer $n\geq 2$, and put $R=\Z/n\Z$. Choose $m$ divisible by $n$ and by every $\expon(G_j)$; for example,
\begin{equation}\label{eq:commonexponent}
 m=\operatorname{lcm}\bigl(n,\expon(G_1),\ldots,\expon(G_t)\bigr).
\end{equation}
For $G\in\calG$ and $\boldsymbol{g}=(g_1,\ldots,g_r)\in G^r$, define $T_{G,\boldsymbol{g}}\subseteq R^r$ by
\begin{equation}\label{eq:targetset}
 \begin{split}
 T_{G,\boldsymbol{g}}=\bigl\{d\in R^r:\;&\text{there are }1\leq k_i\leq m\text{ with }
 k_i\equiv 1-d_i\pmod n,\\[-1mm]
 &g_1^{k_1}\cdots g_r^{k_r}=g_1\cdots g_r\bigr\}.
 \end{split}
\end{equation}

\begin{corollary}\label{cor:dense}
The sets $T_{G,\boldsymbol{g}}$ satisfy the hypotheses on $\mathcal{T}_r$ in Lemma~\ref{lem:sparse}, with constants independent of $r$.
\end{corollary}
\begin{proof}
Let $M=\max_{G\in\calG}|G|$ and $\eta=m^{-(M-1)}$. By Lemma~\ref{lem:fibres}, at least $\eta m^r$ exponent tuples give the target product $g_1\cdots g_r$. The reduction map
\[
 (k_1,\ldots,k_r)\longmapsto(1-k_1,\ldots,1-k_r)\pmod n
\]
has fibres of size exactly $(m/n)^r$. Therefore
\[
 |T_{G,\boldsymbol{g}}|\geq\eta n^r.
\]
There are at most $tM^r$ sets in the family. Taking $k_i=1$ for every $i$ shows that each set contains zero.
\end{proof}

\section{Separating identities and the main theorems}\label{sec:separation}
The quotient construction and the product-fibre estimate give the following separation theorem.

\begin{theorem}[Relative finite-fragment separation]\label{thm:relative}
Let $H$ be a finite abelian group with $|H|\geq 3$, and let $\calG=\{G_1,\ldots,G_t\}$ be a nonempty finite family of finite groups. For every finite set $\Sigma\subseteq\Id(\Pow(H))$, there exist a finite commutative ai-semiring $Q$ and an inequality $\varepsilon$ such that
\[
 Q\models\Sigma,\qquad Q\not\models\varepsilon,
 \qquad \Full{G_j}\models\varepsilon\quad(1\leq j\leq t).
\]
The separating inequality can be chosen so that every word on both sides has the same content. Moreover, with $n=\expon(H)$, the semiring $Q$ may be taken to be a quotient of $\Pow((\Z/n\Z)^r)$ for some positive integer $r$.
\end{theorem}
\begin{proof}
Put $n=\expon(H)$ and $R=\Z/n\Z$. Apply Lemma~\ref{lem:reduction} to $\Sigma$, obtaining a finite family $\calE$ of word inequalities valid in $\Pow(H)$. Form its occurrence blockers and let $\calC$ contain one representative of each affine type from Lemma~\ref{lem:types}. If no blocker is retained, take $\calC=\varnothing$.

Choose $m$ as in \eqref{eq:commonexponent}. By Lemma~\ref{lem:sparse} and Corollary~\ref{cor:dense}, for some sufficiently large $r$ there is a set
\[
 D\subseteq R^r\setminus\{0\},\qquad |D|\leq 2r^2,
\]
which contains no affine copy of a member of $\calC$ and meets every set $T_{G,\boldsymbol{g}}$. In particular $D$ is nonempty.

For $a\in R$, define the one-variable polynomial
\begin{equation}\label{eq:liftpoly}
 \bp_a(x)=\sum_{\substack{1\leq k\leq m\\ k\equiv 1-a\;(\mathrm{mod}\ n)}}x^k.
\end{equation}
There are exactly $m/n$ summands, each with a positive exponent. For $d=(d_1,\ldots,d_r)\in D$, put
\begin{equation}\label{eq:separatorpoly}
 \bq_r=x_1x_2\cdots x_r,\qquad
 \bu_D=\sum_{d\in D}\bp_{d_1}(x_1)\bp_{d_2}(x_2)\cdots \bp_{d_r}(x_r).
\end{equation}
The order of the factors is $1,2,\ldots,r$ in every summand. We claim that
\begin{equation}\label{eq:epsilon}
 \varepsilon:\qquad\bq_r\preceq\bu_D
\end{equation}
holds in $\Full{G}$ for every $G\in\calG$.

Indeed, let $X_1,\ldots,X_r\subseteq G$. If some $X_i$ is empty, then both terms in the inequality have empty value, since every expanded monomial contains every variable with positive exponent. Otherwise take an arbitrary
\[
 g=g_1\cdots g_r\in X_1X_2\cdots X_r,
 \qquad g_i\in X_i.
\]
Choose $d\in D\cap T_{G,\boldsymbol{g}}$ and exponent witnesses $k_i$ from \eqref{eq:targetset}. The word $x_1^{k_1}\cdots x_r^{k_r}$ occurs in the expansion of the $d$-summand of \eqref{eq:separatorpoly}, and
\[
 g=g_1^{k_1}\cdots g_r^{k_r}
   \in X_1^{k_1}\cdots X_r^{k_r}
   \subseteq\bu_D(X_1,\ldots,X_r).
\]
This proves the claim.

Now set $V=R^r$ and $Q=Q_{\calE}(V)$. By Lemmas~\ref{lem:fragmentmodel} and \ref{lem:reduction}, $Q\models\Sigma$. Let $e_1,\ldots,e_r$ be the standard basis of $V$, and put $\beta=e_1+\cdots+e_r$. Evaluate the variables at the quotient classes of the singletons $\{e_i\}$. Before taking the quotient, we have
\begin{equation}\label{eq:evaluation}
 \bq_r(\{e_1\},\ldots,\{e_r\})=\{\beta\},\qquad
 \bu_D(\{e_1\},\ldots,\{e_r\})=K:=\beta-D.
\end{equation}
The second equality holds because all exponents in $\bp_{d_i}$ have the same residue $1-d_i$ modulo $n$.

The set $K$ does not belong to $\calI_{\calE}(V)$. Otherwise it would contain $a+\phi(B)$ for a retained blocker with $0\notin\phi(B)$. The affine bijection $v\mapsto\beta-v$ would then place
\[
 (\beta-a)+(-\phi)(B)
\]
inside $D$, contrary to the avoidance property and the definition of $\calC$.

Also $\beta\notin K$ because $0\notin D$. The class of $K$ is a singleton congruence class, so $K$ and $K\cup\{\beta\}$ have distinct classes, whether or not the latter belongs to the collapsed family. Consequently,
\[
 [K]+[\{\beta\}]=[K\cup\{\beta\}]\ne[K].
\]
Thus $Q\not\models\varepsilon$, while \eqref{eq:epsilon} holds in every $\Full{G_j}$. This proves the theorem.
\end{proof}

\begin{lemma}\label{lem:abeliansubgroup}
Every finite group of order at least three contains an abelian subgroup of order at least three.
\end{lemma}
\begin{proof}
If an element has order greater than two, its cyclic subgroup suffices. Otherwise every element is its own inverse, and $xy=(xy)^{-1}=yx$ for all $x,y$. The group itself is then abelian.
\end{proof}

\begin{proposition}\label{prop:localwitness}
Let $H$ be a finite group with $|H|\geq 3$, and let $G_1,\ldots,G_t$ be finite groups. For every $k\geq 1$, there are a finite commutative ai-semiring $Q_k$ and an inequality $\sigma_k$ such that
\begin{enumerate}[label=\textup{(\roman*)},leftmargin=2.4em]
\item every subsemiring of $Q_k$ generated by at most $k$ elements belongs to $\Var(\Pow(H))$;
\item $Q_k\not\models\sigma_k$;
\item $\Full{H},\Full{G_1},\ldots,\Full{G_t}$ all satisfy $\sigma_k$;
\item every word on either side of $\sigma_k$ has the same content.
\end{enumerate}
\end{proposition}
\begin{proof}
Choose an abelian subgroup $A\leq H$ with $|A|\geq 3$. By Lemma~\ref{lem:finitevariable}, there is a finite set $\Delta_k\subseteq\Id_k(\Pow(A))$ implying all of $\Id_k(\Pow(A))$. Apply Theorem~\ref{thm:relative} with lower group $A$, upper family $\{H,G_1,\ldots,G_t\}$, and finite fragment $\Delta_k$. This yields $Q_k$ and $\sigma_k$ satisfying (ii)--(iv), and $Q_k$ satisfies every identity in $\Id_k(\Pow(A))$.

Let $T$ be a subsemiring of $Q_k$ generated by $a_1,\ldots,a_k$, repeating generators if necessary. If two terms in $x_1,\ldots,x_k$ induce the same term function on $\Pow(A)$, their equality holds in $Q_k$ and hence under $x_i\mapsto a_i$. Therefore this assignment induces a surjective homomorphism from the relatively free $k$-generated semiring of $\Var(\Pow(A))$ onto $T$. It follows that
\[
 T\in\Var(\Pow(A))\subseteq\Var(\Pow(H)),
\]
which proves (i).
\end{proof}

\begin{proof}[Proof of Theorem~\ref{thm:interval}]
Let $\calU=\Var(\Pow(H))$. For each $k\geq 1$, take $Q_k$ and $\sigma_k$ from Proposition~\ref{prop:localwitness}. The inequality $\sigma_k$ holds in every generating semiring of $\calW$, so $Q_k\notin\calW$. Every subsemiring of $Q_k$ generated by at most $k$ elements belongs to $\calU$. Lemma~\ref{lem:local} now gives infinite axiomatic rank for each $\calV$ in~\eqref{eq:maininterval}. Every finite identity basis has a finite bound on its number of variables, so each such $\calV$ is nonfinitely based.
\end{proof}

\begin{proof}[Proof of Corollaries~\ref{cor:zerointerval} and~\ref{thm:classification}]
Take $H=G$ and $t=1$ in Theorem~\ref{thm:interval}. Since $\Pow(G)$ is a subsemiring of $\Full{G}$, this proves Corollary~\ref{cor:zerointerval} and the nonfinite-basis assertion of Corollary~\ref{thm:classification}.

For $|G|=1$, the semiring $\Full{G}$ is the two-element distributive lattice. Relative to the ai-semiring axioms, it has the finite basis
\[
 x^2\approx x,\qquad xy\approx yx,\qquad xy\preceq x.
\]
For $|G|=2$, we have $G\cong C_2$, and a basis for $\Var(\Full{C_2})$ relative to the ai-semiring axioms is
\begin{equation}\label{eq:C2basis}
 x^3\approx x,\qquad xy\approx yx
\end{equation}
by Ren and Zhao~\cite[Lemma~3.7]{rz}. Since the ai-semiring axioms form a finite set, both semirings are finitely based.
\end{proof}

\begin{corollary}\label{cor:finitejoins}
For finite groups $G_1,\ldots,G_t$, the following are equivalent:
\begin{enumerate}[label=\textup{(\roman*)},leftmargin=2.4em]
\item $\Var(\Full{G_1},\ldots,\Full{G_t})$ is finitely based;
\item $\prod_{j=1}^{t}\Full{G_j}$ is finitely based;
\item $|G_j|\leq 2$ for every $j$.
\end{enumerate}
\end{corollary}
\begin{proof}
A finite direct product generates the join of the varieties generated by its factors, since each projection is onto and the product belongs to the join. This proves the equivalence of (i) and (ii). If $|G_j|\geq 3$ for some $j$, apply Theorem~\ref{thm:interval} with $H=G_j$. Otherwise the join is either $\Var(\Full{C_1})$ or $\Var(\Full{C_2})$, because $\Full{C_1}$ embeds into $\Full{C_2}$. Both are finitely based.
\end{proof}

\begin{corollary}\label{cor:mixed}
Let $G_1,\ldots,G_t$ be finite groups, and choose
\[
 S_j\in\{\Pow(G_j),\Full{G_j}\}\qquad(1\leq j\leq t).
\]
If $|G_i|\geq 3$ for some $i$, then every variety in
\[
 [\Var(\Pow(G_i)),\Var(S_1,\ldots,S_t)]
\]
has infinite axiomatic rank and is nonfinitely based. In particular, $\prod_{j=1}^{t}S_j$ is nonfinitely based.
\end{corollary}
\begin{proof}
The lower endpoint is contained in the upper endpoint because $\Pow(G_i)$ embeds into $S_i$. Each $S_j$ embeds into $\Full{G_j}$, so the upper endpoint is contained in $\Var(\Full{G_1},\ldots,\Full{G_t})$. Apply Theorem~\ref{thm:interval} and use the equality between the variety of a finite direct product and the join of the varieties of its factors.
\end{proof}

\section{Zero adjunction and further intervals}\label{sec:extensions}
For an ai-semiring variety $\calU$, put
\[
 \calU^0=\Var\{S^0:S\in\calU\}.
\]
A variety $\calU$ is called \emph{zero-closed} if $\calU^0=\calU$.
The next proposition records the properties of this operation that we need; see also~\cite[Corollaries~2.2 and~2.3]{grz}.

\begin{proposition}\label{prop:zerooperator}
Zero adjunction defines an extensive, monotone, idempotent operation on the lattice of ai-semiring varieties. For every nonempty family $\{S_i:i\in I\}$,
\begin{equation}\label{eq:zerogenerators}
 \bigl(\Var\{S_i:i\in I\}\bigr)^0
 =\Var\{S_i^0:i\in I\}.
\end{equation}
\end{proposition}
\begin{proof}
Extensivity follows from the embedding $S\hookrightarrow S^0$, and monotonicity follows from the definition.

To prove~\eqref{eq:zerogenerators}, set $\calU=\Var\{S_i:i\in I\}$. Consider an inequality $\bq\preceq\bu$ with a word on the left that holds in every $S_i^0$. By Lemma~\ref{lem:zeroidentity}, the term $D_{\bq}(\bu)$ is nonempty and
\[
 S_i\models\bq\preceq D_{\bq}(\bu)\qquad(i\in I).
\]
The latter inequality therefore holds in every $S\in\calU$, so $S^0\models\bq\preceq\bu$ for every such $S$. Lemma~\ref{lem:reduction} gives the assertion for arbitrary identities. The reverse containment follows because $S_i\in\calU$.

Finally, Lemma~\ref{lem:zeroidentity} and
\[
 D_{\bq}\bigl(D_{\bq}(\bu)\bigr)=D_{\bq}(\bu)
\]
show that $S^0$ and $(S^0)^0$ satisfy the same identities. Applying~\eqref{eq:zerogenerators} to the family of zero extensions of members of $\calU$ yields $(\calU^0)^0=\calU^0$.
\end{proof}

\begin{corollary}\label{cor:generalzero}
Let $G_1,\ldots,G_t$ be finite groups and put
\[
 \calW=\Var(\Full{G_1},\ldots,\Full{G_t}).
\]
Let $H$ be a finite group with $|H|\geq 3$, and suppose that
\[
 \Var(\Pow(H))\subseteq\calU\subseteq\calW.
\]
Then $\calU^0\subseteq\calW$, and every variety in $[\calU,\calU^0]$ has infinite axiomatic rank and is nonfinitely based.
\end{corollary}
\begin{proof}
By Proposition~\ref{prop:zerooperator}, the variety $\calW$ is fixed by zero adjunction. Monotonicity gives $\calU^0\subseteq\calW^0=\calW$. Every variety in $[\calU,\calU^0]$ therefore lies in the interval covered by Theorem~\ref{thm:interval}.
\end{proof}

We next construct upper endpoints defined directly by separating inequalities. For a fixed finite group $H$ of order at least three and a fixed finite family $\calG=\{G_1,\ldots,G_t\}$, choose a pair $(Q_k,\sigma_k)$ from Proposition~\ref{prop:localwitness} for each $k\geq 1$. Define
\begin{equation}\label{eq:equationalupper}
 \calW_{H,\calG}
 =\{S:S\text{ is an ai-semiring and }S\models\sigma_k
       \text{ for every }k\geq 1\}.
\end{equation}
This construction parallels the use of equational upper endpoints in~\cite[Theorem~3.4]{grsy}.

\Needspace{9\baselineskip}
\begin{theorem}\label{thm:envelope}
For the choices in~\eqref{eq:equationalupper}, the following hold:
\begin{enumerate}[label=\textup{(\roman*)},leftmargin=2.4em]
\item $\Var(\Full{H},\Full{G_1},\ldots,\Full{G_t})\subseteq\calW_{H,\calG}$;
\item $\calW_{H,\calG}^0=\calW_{H,\calG}$;
\item every variety in $[\Var(\Pow(H)),\calW_{H,\calG}]$ has infinite axiomatic rank and is nonfinitely based.
\end{enumerate}
\end{theorem}
\begin{proof}
Part (i) follows from Proposition~\ref{prop:localwitness}(iii). For (ii), write $\sigma_k$ as $\bq_k\preceq\bu_k$. Every word in $\bu_k$ has content $c(\bq_k)$, so $D_{\bq_k}(\bu_k)=\bu_k$. Lemma~\ref{lem:zeroidentity} shows that $S\models\sigma_k$ implies $S^0\models\sigma_k$. Thus $\calW_{H,\calG}^0\subseteq\calW_{H,\calG}$, and extensivity gives equality.

For (iii), every subsemiring of $Q_k$ generated by at most $k$ elements belongs to $\Var(\Pow(H))$, whereas $Q_k\notin\calW_{H,\calG}$ because it fails $\sigma_k$. Apply Lemma~\ref{lem:local}.
\end{proof}

\begin{corollary}\label{cor:envelopezero}
If $\Var(\Pow(H))\subseteq\calU\subseteq\calW_{H,\calG}$, then every variety in $[\calU,\calU^0]$ has infinite axiomatic rank and is nonfinitely based.
\end{corollary}
\begin{proof}
Monotonicity and Theorem~\ref{thm:envelope}(ii) give $\calU^0\subseteq\calW_{H,\calG}$. The assertion follows from Theorem~\ref{thm:envelope}(iii).
\end{proof}

\section{Examples}\label{sec:examples}
\subsection{Variables occurring only in the upper term}
The inequality
\begin{equation}\label{eq:auxexample}
 x\preceq x^3+y+y^2+xy^2
\end{equation}
holds in $\Pow(C_3)$. To verify it, identify $C_3$ with the additive group of $\F_3$, choose nonempty $X,Y\subseteq\F_3$, and let $a\in X$ and $b\in Y$. If $a=0$, then $a\in X^3$. If $b=0$, then $a\in XY^2$. In the remaining case, $a$ equals either $b$ or $2b$, so $a\in Y\cup Y^2$.

For this inequality, the occurrence module is $\F_3^2$, with one vector for $x$ and one auxiliary vector for $y$. With $\alpha=(1,0)$, its blocker is
\[
 B=\{(1,0),(1,2),(1,1),(0,1)\}.
\]
It avoids zero and has a difference space of dimension two. Thus $|L(B)|=9<27=3^{|B|-1}$. This example exhibits the additional variables accommodated by Section~\ref{sec:quotients}.

\subsection{Prime and composite exponents}
For $H=C_3$, Proposition~\ref{prop:defect} gives affine dependence over $\F_3$, as in the occurrence construction of~\cite{note}. For $H=C_p^d$, the blockers meet kernels of maps into $\F_p^d$. For $H=C_4$, the ambient modules are powers of $\Z/4\Z$ and the relevant estimate is $|L(C)|<4^{|C|-1}$.

The group $C_2$ has a different configuration. The set
\[
 C=\{(1,0),(0,1),(1,1)\}\subseteq\F_2^2
\]
meets the kernel of every map to $C_2$ and satisfies $|L(C)|=4=2^{|C|-1}$. For the target group $C_2\times C_2$, this set is not a blocker, since the identity map has kernel $\{0\}$. Theorem~\ref{thm:interval} applies to $C_2\times C_2$ and to all larger elementary abelian $2$-groups.

\subsection{Ordered products in a nonabelian group}
For $G=S_3$ and its subgroup $H=C_3$, take $n=3$ and $m=6$. The lifting polynomials in~\eqref{eq:liftpoly} are
\[
 \bp_0(x)=x+x^4,\qquad
 \bp_1(x)=x^3+x^6,\qquad
 \bp_2(x)=x^2+x^5.
\]
Each has a single value on singleton inputs in an exponent-three abelian group. Their exponent lifts give the dense sets required for $S_3$, while every product in~\eqref{eq:separatorpoly} retains the order $x_1,\ldots,x_r$. In particular, Theorem~\ref{thm:interval} gives infinite axiomatic rank throughout
\[
 [\Var(\Pow(C_3)),\Var(\Full{S_3})].
\]

\subsection{The infinite cyclic group}
Let $\mathsf{CAI}$ denote the variety of all commutative ai-semirings.

\begin{proposition}\label{prop:infinite}
The semiring $\Full{\Z}$ generates $\mathsf{CAI}$. If $\{m_j:j\in J\}$ is an unbounded set of positive integers, then
\[
 \Var\{\Full{C_{m_j}}:j\in J\}=\mathsf{CAI}.
\]
In particular, both varieties are finitely based.
\end{proposition}
\begin{proof}
The displayed semirings are commutative. Suppose two terms in $x_1,\ldots,x_r$ have different distributive normal forms as finite sets of commutative monomials. Choose an integer $B\geq 2$ larger than every exponent occurring in either normal form, and assign
\[
 x_i\longmapsto\{B^{i-1}\}\subseteq\Z.
\]
A monomial with exponent vector $(a_1,\ldots,a_r)$ evaluates to $\{\sum_i a_iB^{i-1}\}$. Distinct exponent vectors give distinct values by uniqueness of base-$B$ expansion. Thus the two terms have different values in $\Full{\Z}$.

For the cyclic family, choose $m_j$ larger than all the integer monomial values just obtained. Reduction modulo $m_j$ keeps these values distinct and separates the same two terms in $\Full{C_{m_j}}$. Hence in each case the common identities are precisely those of commutative ai-semirings, which have a finite basis.
\end{proof}

\appendix
\section{Expansions by constants}\label{app:constants}
For a group $G$ and a choice $\mathscr C\subseteq\{0,1\}$, let $\Full{G}_{\mathscr C}$ be the expansion of $\Full{G}$ in which the chosen constants are interpreted by
\[
 0=\varnothing,\qquad 1=\{1_G\}.
\]

\begin{proposition}\label{prop:constants}
If $G$ is a finite group with $|G|\geq 3$, then $\Full{G}_{\mathscr C}$ is nonfinitely based for every $\mathscr C\subseteq\{0,1\}$.
\end{proposition}
\begin{proof}
Choose an abelian subgroup $H\leq G$ of order at least three, and let $\Sigma$ be a finite set of identities of $\Full{G}_{\mathscr C}$. They also hold in $\Full{H}_{\mathscr C}$. Normalize the terms by distributivity and the laws of the named constants. A normalized term is either $0$, when this symbol is present, or a nonempty sum of words; the empty word is allowed when $1$ is named. A valid identity cannot equate $0$ with a nonempty polynomial, as is seen by assigning $\{1_H\}$ to every variable.

Reduce identities between nonempty polynomials to word inequalities as in Lemma~\ref{lem:reduction}. In each inequality, assign the empty set to variables outside the content of its lower word. This gives a finite family $\calE$ of valid inequalities whose upper words use only variables from the lower word. The family $\calE$, together with the normalization laws, implies $\Sigma$. If the lower word is empty, the restricted upper term contains an empty word and the inequality is immediate; discard it.

For the remaining inequalities use the occurrence construction of Section~\ref{sec:quotients}. An empty upper word contributes the singleton $\{0\}$ in the additive module, which is preserved by all module homomorphisms. Thus the blocker property, the cardinality estimate, and the quotient argument give a finite quotient $Q_{\calE}(V)$ satisfying $\calE$ on nonempty subset representatives. Its multiplicative identity is the class of $\{0\}$.

Adjoin a new zero to $Q_{\calE}(V)$ and interpret the chosen constants by this zero and the class of $\{0\}$. Each restricted inequality remains valid: an assignment of the new zero to a variable of the lower word makes that word zero; all other assignments are covered by the quotient argument. Consequently the expanded semiring $Q_{\calE}(V)^0$ satisfies $\Sigma$.

Choose the dimension of $V$ and the separating inequality as in Theorem~\ref{thm:relative}, with upper family $\{G\}$. This inequality holds in $\Full{G}_{\mathscr C}$ and fails in $Q_{\calE}(V)$ under singleton assignments. The same failure holds in its zero extension. Thus every finite fragment $\Sigma$ has a model failing another identity of $\Full{G}_{\mathscr C}$, proving the assertion.
\end{proof}

\end{document}